\documentclass[12pt]{amsart}

\usepackage{amsthm}

\usepackage{geometry}
\usepackage{dsfont} 
\newtheorem{Theorem2}{Theorem}[section]

\newcommand\phantomarrow[2]{%
  \setbox0=\hbox{$\displaystyle #1\to$}%
  \hbox to \wd0{%
    $#2\mapstochar
     \cleaders\hbox{$\mkern-1mu\relbar\mkern-3mu$}\hfill
     \mkern-7mu\rightarrow$}%
  \,}

\newtheorem{Theorem}{Theorem}

\newtheorem{Claim}{Claim}

\newtheorem{Corollary}{Corollary}[section]

\newtheorem{Definition}{Definition}[section]
\newtheorem{Example}{Example}

\newtheorem{Lemma}{Lemma}[section]

\newtheorem{Proposition}{Proposition}[section]
\newtheorem{Remark}{Remark}

\newtheorem{Question}{Question}

\newcommand{\nocontentsline}[3]{}
\newcommand{\tocless}[2]{\bgroup\let\addcontentsline=\nocontentsline#1{#2}\egroup}
\usepackage{xargs}  
\usepackage[pdftex,dvipsnames]{xcolor}
\usepackage[colorinlistoftodos,prependcaption,textsize=tiny]{todonotes}
\newcommandx{\question}[2][1=]{\todo[linecolor=red,backgroundcolor=orange!25,bordercolor=red,#1]{#2}}
\newcommandx{\change}[2][1=]{\todo[linecolor=blue,backgroundcolor=blue!25,bordercolor=blue,#1]{#2}}
\newcommandx{\comment}[2][1=]{\todo[linecolor=yellow,backgroundcolor=yellow!25,bordercolor=OliveGreen,#1]{#2}}
\newcommandx{\improvement}[2][1=]{\todo[linecolor=Plum,backgroundcolor=Plum!25,bordercolor=Plum,#1]{#2}}
\newcommandx{\thiswillnotshow}[2][1=]{\todo[disable,#1]{#2}}

\usepackage{color}   
\usepackage{hyperref}
\hypersetup{
    colorlinks=true, 
    linktoc=all,     
    linkcolor=blue,  
    citecolor=blue,
    filecolor=blue,
    urlcolor=blue
}

\usepackage[latin1]{inputenc}

\usepackage[T1]{fontenc}
\usepackage{csquotes}
\usepackage[style=numeric,defernumbers,backend=bibtex]{biblatex}
\usepackage{hyperref}
\usepackage{nameref}
\defbibfilter{other}{
  not type=article
  and
  not type=book
  and
  not type=collection
}

\begin{document}

\title[On the Torelli group action on compact character varieties]{On the Torelli group action on compact character varieties}
\author{YOHANN BOUILLY}

\begin{abstract}
The aim of this article is to prove that the Torelli group action on some $G$-character varieties is ergodic. The case $G=\mathrm{SU}(2)$ was obtained by Funar--Marché. We propose another proof of it based on independent methods which extend to a proof of the ergodicity of the Torelli group action on the $G$-character variety for $G$ a connected, semi-simple and compact Lie group.
\end{abstract}

\maketitle

\section{Introduction}

Let $\Sigma$ be a compact, connected, oriented and closed surface of genus $g\geq 2$. Denote $\Gamma$ its fundamental group and let $G$ be a semi-simple, connected and compact Lie group with Lie algebra $\mathfrak{g}$ and adjoint representation $\operatorname{Ad}:G\rightarrow\mathrm{GL}(\mathfrak{g})$. We define $\mathrm{Hom}(\Gamma,G)$ to be the set of homomorphisms $\rho:\Gamma\rightarrow G$, on which the group $G$ acts by conjugation. We denote by $\mathrm{Z}_G(\rho)$ the centralizer of $\rho$, i.e the set of elements of $G$ which commute with all the $\rho(\gamma)$, $\gamma\in\Gamma$. This centralizer $\mathrm{Z}_G(\rho)$ is the stabilizer of $\rho$ for the conjugation action of $G$.

\begin{Definition}
The $G$-\emph{character variety} $\mathcal{X}(\Gamma,G)$ is the GIT-quotient \[\mathrm{Hom}(\Gamma,G)/\!/ G.\]
\end{Definition}
In the cases we consider, the set $\mathcal{X}(\Gamma,G)$ contains a dense open set, which is the set of classes of representations which have a discrete centralizer. Since a semi-simple Lie group has discrete center the previous condition makes sense. This set of regular points is denoted by $\mathcal{M}(\Gamma,G)$.
Goldman proved in \cite{zbMATH03916723} that it carries a sympletic measure. The symplectic form, we will denote by $\omega_{\mathrm{G}}$,  is obtained via the following construction. For $[\rho]$ a smooth point, the tangent space $T_{[\rho]}\mathcal{M}(\Gamma,G)$ is identified with the first group cohomology:
 \[\mathrm{H}^1(\Gamma,\frak{g}_{\rho})=\frac{\mathrm{Z}^1(\Gamma,\frak{g}_{\rho})}{\mathrm{B}^1(\Gamma,\frak{g}_{\rho})}\]
 
\noindent where $\mathrm{Z}^1(\Gamma,\frak{g}_{\rho})$ is the set of maps $u:\Gamma\rightarrow \mathfrak{g}$ such that for all $\gamma_1,\gamma_2\in\Gamma$,\[u(\gamma_1\gamma_2)=u(\gamma_1)+  \operatorname{Ad}(\rho(\gamma_1))u(\gamma_2)\]
\noindent and $\mathrm{B}^1(\Gamma,\frak{g}_{\rho})$ be the set of maps of the form $\gamma\mapsto x - \operatorname{Ad}(\rho(\gamma))x$ for $x\in\mathfrak{g}$.
 
 We define the second cohomology group $\mathrm{H}^2(\Gamma,\mathbf{R})$ as the quotient : 
 \[\frac{\mathrm{Z}^2(\Gamma,\mathbf{R})}{\mathrm{B}^2(\Gamma,\mathbf{R})}\]
 
 \noindent where $\mathrm{Z}^2(\Gamma,\mathbf{R})$ is the set of maps $u:\Gamma^2\rightarrow \mathbf{R}$ which verify the cocycle relation 
 \[u(\gamma_2,\gamma_3)-u(\gamma_1 \gamma_2,\gamma_3)+u(\gamma_1,\gamma_2\gamma_3 )- u(\gamma_1,\gamma_2)=0\]
 for all $\gamma_1,\gamma_2,\gamma_3\in\Gamma$ and $\mathrm{B}^2(\Gamma,\mathbf{R})$ is the set of maps $\Gamma^2\rightarrow\mathbf{R}$ of the form \[(\gamma_1,\gamma_2) \mapsto
 v(\gamma_2)+v(\gamma_1)-v(\gamma_1\gamma_2).\]
 It is well known that the second cohomology group $\mathrm{H}^2(\Gamma,\mathbf{R})$ is isomorphic to $\mathbf{R}$.
 
 We hence define $\omega_{\mathrm{G}}$ by the formula : 
\begin{align*}
  \omega_{\mathrm{G} [\rho]}  \colon &\mathrm{H}^1(\Gamma,\frak{g}_{\rho}) ^2 \to \mathrm{H}^2(\Gamma,\mathbf{R})\cong\mathbf{R}\\
  &\phantomarrow{\mathrm{H}^1(\Gamma,\frak{g}_{\rho}) ^2 }{(u,v)} [ (\gamma_1,\gamma_2)\mapsto \langle u(\gamma_1), \operatorname{Ad}(\rho(\gamma_1)) v(\gamma_2)\rangle].
\end{align*}
\noindent where $\langle\cdot ,\cdot \rangle$ is the Killing form of $\mathfrak{g}$. It defines a symplectic measure $\mu$ on the character variety.

The mapping class group $\mathrm{Mod}^+(\Sigma)$ is the group of isotopy classes of orientation-preserving diffeomorphisms of $\Sigma$. It acts on the $G$-character variety via :
\[[\psi]\cdot [\rho]=[\gamma\mapsto\rho(\psi^{-1}_{*}\gamma)]\]
and preserves $\mathcal{M}(\Gamma,G)$ and its symplectic structure.
Naturally $\mathrm{Mod}^+(\Sigma)$ acts on the first homology group $\mathrm{H}_{1}(\Sigma,\mathbf{Z})$ of $\Sigma$. The kernel of this action is called the Torelli group and denoted by $\mathrm{Tor}(\Sigma)$.

William Goldman and Eugene Xia proved in \cite{zbMATH06077159}  the following theorem :

\begin{Theorem2}
\label{Theorem 1.1}
The mapping class group acts ergodically on the character variety $\mathcal{M}(\Gamma,\mathrm{SU}(2))$ with respect to the symplectic measure.
\end{Theorem2}

There are however stronger dynamical properties : Louis Funar and Julien Marché proved recently in \cite{zbMATH06231013} that the actions on $\mathcal{M}(\Gamma,\mathrm{SU}(2))$ of the Johnson groups $\mathcal{K}_n$, defined by $\mathcal{K}_0=\mathrm{Mod}^+(\Sigma)$ and $\mathcal{K}_{n+1}=[\mathcal{K}_n,\mathcal{K}_n]$, are ergodic.

As a consequence :

\begin{Corollary}
The Torelli group acts ergodically on the  $\mathrm{SU}(2)$-character variety with respect to the symplectic measure.
\end{Corollary}

Their proof use the local geometry of the character variety at the trivial representation and Taylor expansions of the trace functions.

The aim of this paper is to propose a new proof of the ergodicity of the Torelli group action on $\mathcal{M}(\Gamma,\mathrm{SU}(2))$. The tools and the strategy we use will be adapted for the cases $G=\mathrm{SU}(n)$, for $n\geq 2$ :

\begin{Theorem}
\label{Theorem 1}
The Torelli group acts ergodically on the  $\mathrm{SU}(n)$-character variety $\mathcal{M}(\Gamma,\mathrm{SU}(n))$ with respect to the symplectic measure.
\end{Theorem}

Doug Pickrell and Eugene Xia generalized the result of \cite{zbMATH06077159} in \cite{zbMATH01807864} :

\begin{Theorem2} 
\label{Theorem 1.2}
Let $G$ be a connected and compact Lie group. Then the mapping class group acts ergodically on each connected component of the $G$-character variety $\mathcal{M}(\Gamma,G)$ with respect to the measure on $\mathcal{M}(\Gamma,G)$ induced by the Haar measure on $G$.
\end{Theorem2}

We so generalize :

 \begin{Theorem}
 \label{Theorem 2}
Let $G$ be a connected, semi-simple and compact Lie group. Then the Torelli group $\mathrm{Tor}(\Sigma)$ acts ergodically on each connected component of $\mathcal{M}(\Gamma,G)$ with respect to the Goldman symplectic measure.
\end{Theorem}

As corollary, replacing $G$ by a finite product $G\times\dots\times G$, we obtain : 

\begin{Theorem}
Let $G$ be a group verifying the hypothesis of theorem~\ref{Theorem 2}. Then, for all $k\geq 1$, the Torelli group $\mathrm{Tor}(\Sigma)$ acts ergodically on each connected components of the product $\mathcal{M}(\Gamma,G)^k$. In particular on these components, the action of the Torelli group is weakly mixing.
\end{Theorem}

\section*{Opening}

Following the works of Funar-Marché, the natural continuation of this article is about the ergodic action of the Johnson subgroups on the character varieties we consider : 

\begin{Question}
Let $\mathcal{M}(\Gamma,G)$ be a character variety considered in this article. Is the action of the Johnson subgroups ergodic on it ? Is there a index $n>0$ such that for all $i\leq n$ the action of $\mathcal{K}_i$ is ergodic and not for $i\geq n+1$ ?
\end{Question}

Julien Marché and Maxime Wolff proved in \cite{zbMATH06556671} that the mapping class group acts ergodically on some subspaces of the exotic components of the $\mathrm{PSL}_2(\mathbf{R})$-character variety. The following question is a natural problem about the non-compact cases : 

\begin{Question}
Is the Torelli group action on the subspaces introduced by Marché and Wolff ergodic ?
\end{Question}

In order to give a complete description of the situation, we may wonder if theorem~\ref{Theorem 2} holds by dropping the semi-simple condition. Namely, we ask :
\begin{Question}
Is the theorem~\ref{Theorem 2} true if $G$ is not assumed to be semi-simple ?
\end{Question}

In the abelian case, it is observed in \cite{YBandGF}, that the action is trivial.

\section*{Aknowledgements}

I am redebted to my advisor Olivier Guichard for our discussions about this article and for his permanent support and advices. I would like to thank Julien Marché and Maxime Wolff for their interest about this work and Gianluca Faraco for his remarks about this project.

\section{Background}

\subsection{The Torelli group}

For more details on this part, see \cite{zbMATH05960418}. The \textit{mapping class group} $\mathrm{Mod}^+(\Sigma)$ of the surface $\Sigma$ is the quotient of the set of positive diffeomorphisms of $\Sigma$ quotiented by isotopy relation. It means that a mapping class is a class $[f]=\{g\in\text{Diff}^+(\Sigma)\mid g\text{ and } f \text{ are isotopic}\}$.\\ 
It is generated by the $3g-1$ Dehn twists $T_{a_1},\dots,T_{b_g},T_{d_1},\dots,T_{d_{g-1}}$ where the curves $a_i , b_i$ are given by the presentation : 
\[\Gamma=\langle a_1,b_1,\dots,a_g,b_g\mid [a_1,b_1]\cdots[a_g,b_g]=1\rangle\] \noindent and the curves $d_i$ are the products $a_i^{-1} a_{i+1}$.\\ 
The first homology group $\mathrm{H}_1(\Sigma,\mathbf{Z})\cong \mathbf{Z}^{2g}$ is freely generated by the curves $[a_1],\dots,[b_g]$ and is a lattice in the real homology group $\mathrm{H}_1(\Sigma,\mathbf{R})$.

\begin{Definition}
The \emph{Torelli group} $\mathrm{Tor}(\Sigma)$ is the kernel of the action of the mapping class group on $\mathrm{H}_1(\Sigma,\mathbf{Z})$.
\end{Definition}

An explicit generating set of the Torelli group is the set :  
\[\{ T_c , T_aT_b ^{-1}\mid c\emph{ \text{ separating curve and }} a,b \emph{\text{ are cohomologuous curves}} \}.\]

We make an important remark for the following.
\begin{Remark}
 if $c_1,\dots,c_{\ell}$ are the boundary components of a subsurface of $\Sigma$, then the product of the Dehn twists $T_{c_1},\dots, T_{c_{\ell}}$ acts trivially on $\mathrm{H}_1(\Sigma,\mathbf{R})$ and hence is in the Torelli group.
 \end{Remark}

\subsection{Borel cross sections}

Let $X$ be a topological set on which a topological group $H$ acts. The set $X$ is endowed with its Borelian $\sigma$-algebra and a Borelian measure $\mu$. A function $X_1\rightarrow X_2$ between two measured sets is bimeasurable if it is, measurable, invertible and has a measurable inverse. If such a function exists, we say that $X_1$ and $X_2$ are bimeasurable. The quotient set $X\slash H$ carries the quotient topology such that the canonical projection $X\rightarrow X\slash H$ is continuous.
A \emph{Borel cross section} is a subset $\mathcal{S}\subset X$ which intersects the orbits $H.x$, for each element $x\in X$, exactly once. Edward Effros gives a proof of the following theorem in \cite{zbMATH03244880}. We state a more restrictive result but sufficient for our purpose.

\begin{Theorem2}
\label{Theorem 2.1}
If $X$ and $H$ are separable, complete, metrizable and locally compact and if $H$ acts continuously on $X$, then the following condition are equivalent : 
\begin{itemize}
\item Each orbit is locally closed.
\item Each orbit is locally compact.
\item There exists a Borel cross section $\mathcal{S}$ for the orbits of $H$ in $X$.
\end{itemize}

\end{Theorem2}

James Bondar gives this key corollary in \cite{zbMATH03538627}: 

\begin{Theorem2}
\label{Theorem 2.2}
If $X$ and $H$ verify the hypothesis of the previous theorem and if furthermore the action of $H$ on $X$ is free and one of the condition of the previous statement holds, then the Borel cross section $\mathcal{S}$ is bimeasurable to the quotient space $X\slash H$ and for every measurable function $f:X\rightarrow \mathbf{R}$ :

\[\int  _{X} f d\mu =\int  _{\mathcal{S}} \bigg( \int  _{H} f(h.s) d\mu_H (h)\bigg) d\mu_{\mathcal{S}}(s)\] 

\noindent where $\mu_{\mathcal{S}}$ is some Borelian measure on $\mathcal{S}$ and $\mu_H$ is the Haar-measure of $H$.

\end{Theorem2}

\section{Ergodicity for the case of $\mathrm{SU}(2)$}

The strategy in proving the ergodicity of the Torelli group action is to find a full measure subset of the character variety on which every measurable and $\mathrm{Tor}(\Sigma)$-invariant function is invariant by the mapping class group action.

\subsection{The mapping class group action on the $\mathrm{SU}(2)$-characters}

In the rank $1$ case we consider, the Lie algebra $\frak{su}(2)$ of $\mathrm{SU}(2)$ is the Lie algebra of traceless skew-Hermitian complex $2\times 2$-matrices.
Let $f:\mathrm{SU}(2)\rightarrow [-2,2]$ denote the trace function. Its \emph{variation function} $F$ is defined as the unique function $F:\mathrm{SU}(2)\rightarrow \frak{su}(2)$ such that, for all $x\in \mathrm{SU}(2)$ and $X\in\frak{su}(2)$,
\[ \frac{d}{dt}_{|t=0}f (x.\exp(tX))=\langle F(x),X\rangle.\]

Goldman proved, in \cite{zbMATH04005191}, that :
\[F(x)=x-\frac{\mathrm{tr}(x)}{2}\mathrm{id}.\]  

Following \cite{zbMATH06077159}, we let : 
\begin{align*}
 \zeta^t \colon & \mathrm{SU}(2)\to  \mathrm{SU}(2)\\
  &\phantomarrow{ \mathrm{SU}(2) }{x}  \exp (t F(x)).
\end{align*}

For $x\in\mathrm{SU}(2)$, the map $t\mapsto \zeta^t(x)$ is a one-parameter subgroup of $\mathrm{SU}(2)$.

\subsubsection{Separating curve}

If a curve $\alpha$ is separating (i.e $\Sigma\setminus\alpha$ isn't connected), then $\Sigma\setminus\alpha$ is the disjoint union $\Sigma_1 \sqcup \Sigma_2$ and the fundamental group $\Gamma$ is the amalgamated product :
\[\pi_1\Sigma_1 \underset{\langle\alpha\rangle}{\ast}\pi_1\Sigma_2.\]

The data of two representations $\rho_1:\pi_1\Sigma_1\rightarrow \mathrm{SU}(2)$ and $\rho_2:\pi_1\Sigma_2\rightarrow \mathrm{SU}(2)$ such that
\[\rho_1(\alpha)=\rho_2(\alpha)\]
allows to construct a unique representation $\rho:\Gamma\rightarrow \mathrm{SU}(2)$ defined by $\rho_{|\pi_1\Sigma_1}=\rho_1$ and $\rho_{|\pi_1\Sigma_2}=\rho_2$.

We let the twist flow $\xi_{\alpha}^t :\mathcal{M}(\Gamma,\mathrm{SU}(2))\rightarrow \mathcal{M} (\Gamma,\mathrm{SU}(2))$ be defined by : 
\[ 
\xi_{\alpha}^t\rho(\gamma)= 
\left\{\begin{array}{l l}
\rho(\gamma) & \text{if }\gamma\in\pi_{1}(\Sigma_1)\\
\zeta^t (\rho(\alpha))\rho(\gamma)\zeta^{-t} (\rho(\alpha))
& \text{if }\gamma\in\pi_1(\Sigma_2)\\
\end{array}\right..
\]

This flow is well defined because $\zeta^t (\rho(\alpha))$ is the exponential of a polynomial in $\rho(\alpha)$ and thus commutes with $\rho(\alpha)$.  

\subsubsection{Non-separating curve}

If $\alpha$ is non-separating (i.e if $\Sigma | \alpha$ is connected), the fundamental group $\Gamma$ is the HNN-extension :
\[\bigg( \pi_1(\Sigma |\alpha)*\langle\beta\rangle \bigg)\big\slash \langle\beta \alpha_- \beta^{-1}\alpha_+^{-1}\rangle \]

\noindent where $\alpha_\pm$ represent the boundary components of $\Sigma | \alpha$.
Hence the data of a representation 
\[\rho_0:\pi_1(\Sigma |\alpha)\rightarrow\mathrm{SU}(2) \]
and a matrix $B\in\mathrm{SU}(2)$ such that 
\[B\rho_0(\alpha_-)B^{-1}=\rho_0(\alpha_+)\]
\noindent defines a unique representation $\rho:\Gamma\rightarrow \mathrm{SU}(2)$ such that $\rho_{|\pi_1(\Sigma |\alpha)}=\rho_0$ and $\rho(\beta)=B$.

We define the flow $\xi_{\alpha}^t :\mathcal{M} (\Gamma,\mathrm{SU}(2))\rightarrow \mathcal{M} (\Gamma,\mathrm{SU}(2))$ be defined by : 
\[ 
\xi_{\alpha}^t\rho(\gamma)= 
\left\{\begin{array}{l l}
\rho(\gamma) & \text{if }\gamma\in\pi_1(\Sigma |\alpha)\\
\zeta^t (\rho(\alpha))\rho(\beta)& \text{if }\gamma=\beta\\
\end{array}\right.
.
\]

This flow is well defined since it verifies the relation :
\[\xi_{\alpha}^t\rho(\beta)\xi_{\alpha}^t\rho(\alpha_-)\xi_{\alpha}^t\rho(\beta)^{-1}=\xi_{\alpha}^t\rho(\alpha_+).\]

\subsubsection{The Dehn twists and the flows}
For a simple and closed curve $\alpha$, let the trace function $f_{\alpha}: \mathcal{M}(\Gamma,\mathrm{SU}(2))\rightarrow [-2,2]$ associated with the curve $\alpha$, i.e. 
\[f_{\alpha}([\rho])=\mathrm{tr} (\rho (\alpha)).\]

In \cite{zbMATH04005191}, Goldman proved that these flows are the Hamiltonian flows of the trace functions. That mean :
\[df_{\alpha}X= \omega_G \bigg(\frac{d}{dt}_{|t=0}\xi_{\alpha}^t,X\bigg)\]

\noindent for all $X\in T\mathcal{M} (\Gamma,\mathrm{SU}(2))$.

If $x\in \mathrm{SU}(2)$, then there exists $g\in\mathrm{SU}(2)$ such that 
$x=g\begin{pmatrix}
e^{i\theta} & 0 \\
0 & e^{-i\theta}
\end{pmatrix}g^{-1}$ with $\theta=\cos^{-1}\big(\frac{f(\rho(\alpha))}{2}\big)$. 
We then compute $F(x)=x-\frac{f(x)}{2}\mathrm{id}$ which we can write: 
\[F(x)=g\begin{pmatrix}
2i\sin(\theta) & 0 \\
0 & -2i\sin(\theta)
\end{pmatrix}g^{-1}\] and so, by definition :
\[\zeta^t (x)=g\begin{pmatrix}
e^{2it\sin(\theta)} & 0 \\
0 & e^{-2it\sin(\theta)}
\end{pmatrix}g^{-1}.\]

In particular, if $x=\pm\mathrm{id}$, then for all $t\in\mathbf{R}$ we have $\zeta ^t (x)=\mathrm{id}$ and for $x\neq \pm \mathrm{id}$, we have the equality $\zeta ^t (x)=\mathrm{id}$ if and only if $t\in \frac{2\pi}{\sin(\theta)}\mathbf{Z}$, and :

\begin{Lemma}
\label{Lemma 3.1}
If $x\neq \pm \mathrm{id}$, then $x$ belongs to the one-parameter subgroup $\{\zeta^t (x)\}_{t\in\mathbf{R}}$ and more precisely $x=\zeta^{s(x)} (x)$ for : 
\[s(x)=\frac{\theta}{2\sin(\theta)}.\]
Moereover, for $x\in \mathrm{SU}(2)$ such that $\theta \notin \pi\mathbf{Q}$, the subgroup $\langle x\rangle$ is dense in the circle $ \{\zeta^t (x)|t\in\mathbf{R}\}\cong \mathbf{S}^1$ and acts ergodically on it with respect to the Lebesgue measure.
\end{Lemma}

We hence remark that the Dehn twist $T_{\alpha}$ acts on $\rho$ with the relation : 
\[T_{\alpha}.\rho = \xi_{\alpha}^{s(\rho(\alpha))}\rho,\]

\noindent and if $\theta$ is irrational, then the orbit $\langle T_{\alpha}\rangle.\rho$ is dense in the circle defined by the orbit $\{ \xi_{\alpha}^{t}\rho\}_{t\in\mathbf{R}}$.
The Hamiltonian flow gives an action of the circle $\mathrm{U}_{\alpha}:=\mathbf{S}^1$ on the subspace of  the character variety consisting of classes of representations $[\rho]$ such that $\rho(\alpha)\neq\pm \mathrm{id}$.

We will propose another version of these flows in the case of $G=\mathrm{SU}(n)$, for $n\geq 3$, in section $5$ and the more general case is treated in section $6$.

\subsection{ Ergodicity of translation actions}

We give here key results whose statements allow to find a condition on a representation and some curves to have a property of density and ergodicity of the $\mathbf{Z}$-action of a Dehn twists composition, similarly to the lemma~\ref{Lemma 3.1}.\\
We remark that if two simple curves $c_1$ and $c_2$ are disjoint, then the flows $\xi_{c_1}^t$ and $\xi_{c_2}^{s}$ commute. Hence, under this assumption, the actions of these flows on a representation $\rho$ give a topological torus orbit $\{\xi_{c_1}^{t}.\xi_{c_2}^s\rho\}_{t,s\in\mathbf{R}}$ obtained by the action of $\mathrm{U}_{c_1}\times\mathrm{U}_{c_2}$ on the characters which are not $\pm \mathrm{id}$ evaluated in $c_1,c_2$.

 \begin{Lemma}
 \label{Lemma 3.2}
 Let $[\rho]$ be a class of representations in $ \mathcal{M} (\Gamma,\mathrm{SU}(2))$ and suppose that there exist $c_1,\dots, c_{\ell}$ be simple closed curves of $\Sigma$ which are pairwise disjoint and such that 
 \[\pi, \theta_1=\cos^{-1}\bigg(\frac{f(\rho(c_1))}{2}\bigg),\dots, \theta_{\ell}=\cos^{-1}\bigg(\frac{f(\rho(c_{\ell}))}{2}\bigg)\] 
are linearly independent over $\mathbf{Q}$.
If we denote $h=T_{c_1} \dots T_{c_{\ell}}$, then the action of $h$ on the orbit $\mathrm{U}_{c_1}\times\dots\times\mathrm{U}_{c_{\ell}}.[\rho]$ is ergodic with respect to the Lebesgue measure on this torus orbit.
\end{Lemma}

In particular, almost every orbit for the action of $h$ is dense in the topological torus \[\mathrm{U}_{c_1}\times\dots\times\mathrm{U}_{c_{\ell}}.[\rho].\]

\begin{Remark}
\label{Remark 3}
With the notations of the lemma~\ref{Lemma 3.2} and if $[\rho]$ verifies the hypothesis of it, then for all $i=1,\dots,\ell$, the matrix $\rho(c_i)$ is different to $\pm \mathrm{id}$, then the torus orbit $\mathrm{U}_{c_1}\times\dots\times\mathrm{U}_{c_{\ell}}.[\rho]$ is the torus obtained as the quotient: 

\[\mathbf{R}^{\ell}\slash \Lambda\]

\noindent where $\Lambda$ is the lattice $\frac{2\pi}{\sin\theta_1}\mathbf{Z}\oplus\dots\oplus \frac{2\pi}{\sin\theta_{\ell}}\mathbf{Z}$.
By definition of the torus $\mathbf{R}^{\ell}\slash \Lambda$, the action of $\mathrm{U}_{c_1}\times\dots\times\mathrm{U}_{c_{\ell}}$ on the character variety is free.

\end{Remark}
The following is a classical result we need to prove the lemma~\ref{Lemma 3.2} and whose no proof will be given (see \cite{zbMATH00776267} for more details). We denote by $\mathbf{T}$ the torus $\mathbf{R}\slash 2\pi\mathbf{Z}$. This notation is justified by the fact that 
\[\mathbf{T}\cong \bigg\{\begin{pmatrix}
e^{i\theta} & 0 \\
0 & e^{-i\theta}
\end{pmatrix}\mid \theta\in\mathbf{R}\bigg\}\]
\noindent is a maximal torus of $\mathrm{SU}(2)$.
\begin{Lemma}
\label{Lemma 3.3}
Let $t=(t_1,\dots,t_{\ell})\in\mathbf{R}^{\ell}$ such that $t_1,\dots, t_{\ell},\pi$ are linearly independent over $\mathbf{Q}$ and let $f_t:\mathbf{T}^{\ell}\rightarrow \mathbf{T}^{\ell}$ be the translation of vector $t$. Then the action of $\langle f_t \rangle$ on $\mathbf{T}^{\ell}$ is ergodic with respect to the Lebesgue measure.
\end{Lemma}

\begin{proof}[Proof of the lemma~\ref{Lemma 3.2}]
We remark that the Dehn twists $T_{c_k}$ act as a translation of the torus orbit $\mathrm{U}_{c_1}\times\dots\times\mathrm{U}_{c_{\ell}}.[\rho]$. The orbit $h^{\mathbf{Z}}.[\rho]$ is the orbit $\langle\xi_{c_1}^{t_1}\dots,\xi_{c_{\ell}}^{t_{\ell}} \rangle.[\rho]$ with for all $k\in\{1,\dots,\ell\}$:

\[t_k=\frac{\theta_k}{2\sin(\theta_k)}.\]

For such a  $t_k$, the action of $\xi_{c_k}^{t_k}$ is given by the multiplication by a matrix conjugated to 
\[\begin{pmatrix}
e^{i\theta_k} & 0 \\
0 & e^{-i\theta_k}
\end{pmatrix}.\]
We deduce from this that the action of $h$ on $\mathrm{U}_{c_1}\times\dots\times\mathrm{U}_{c_{\ell}}.[\rho]$  is given by the translation of the vector $(\theta_1,\dots,\theta_{\ell})$.
Then the lemma~\ref{Lemma 3.3} shows that this action is ergodic with respect to the Lebesgue measure since $\pi,\theta_1,\dots,\theta_{\ell}$ are linearly independent over $\mathbf{Q}$.
\end{proof}

\subsection{A full measure set}

A multicurve $m$ is the union of a finite number of simple, closed and pairwise disjoint curves. Let denote by $MC(\Sigma)$ the set of mutlicurves $m$ such that all the curves of $m$ are simple, closed and non-separating and such that $m$ is the boundary of a a subsurface of $\Sigma$ and by $MC_0(\Sigma)$ its subset of multicurves bounding a pair of pant in the surface $\Sigma$.

\begin{Definition}
Let $m=c_1\cup\dots\cup c_{\ell}\in MC(\Sigma)$ be a mutlicurve . A class $[\rho]\in\mathcal{M} (\Gamma,\mathrm{SU}(2))$ \emph{verifies the condition} $(M_m)$ if the real numbers :  
\[\pi,\theta_1= \cos^{-1}\bigg(\frac{\mathrm{tr}(\rho(c_1))}{2}\bigg),\dots,\theta_{\ell}=\cos^{-1}\bigg(\frac{\mathrm{tr}(\rho(c_{\ell}))}{2}\bigg)\] 
are linearly independant over $\mathbf{Q}$.

\end{Definition}

Following the previous definition, for $m\in MC_0(\Sigma)$, we let the set :
\[\mathcal{M}_m(\Gamma,\mathrm{SU}(2))=\{[\rho]\in \mathcal{M} (\Gamma,\mathrm{SU}(2))\mid [\rho] \text{ statisfies the condition } (M_m)\}\]

The aim of this section is to prove the proposition :

\begin{Proposition}
\label{Proposition 3.2}
For $m\in MC_0(\Sigma)$, the set $\mathcal{M}_m(\Gamma,\mathrm{SU}(2))$ has full measure in $\mathcal{M} (\Gamma,\mathrm{SU}(2))$.
\end{Proposition}

For a curve $\gamma$, the angle of $\rho(\gamma)$, expressed with the formula  \[\cos^{-1} \bigg(\frac{f_{\gamma}}{2}\bigg),\]
defines a function $\theta_{\gamma} : \mathcal{M} (\Gamma,\mathrm{SU}(2))\rightarrow \mathbf{S}^1$.
To simplify the notations, as previously, for a multicurve $m=c_1\cup c_2\cup c_3\in MC_0(\Sigma)$, we will denote by $\theta_1$, $\theta_2$ and $\theta_3$ the functions $\theta_{c_1}$, $\theta_{c_2}$ and $\theta_{c_3}$.  

We will need the following lemma : 
\begin{Lemma}
\label{Lemma 3.4}
Let $[\rho]$ be a class of representations in $ \mathcal{M} (\Gamma,\mathrm{SU}(2))$, i.e. $\rho$ has discrete centralizer, $m=c_1\cup c_2 \cup c_3$ be a pair of pant such that $\rho(c_i)\neq\pm \mathrm{id}$ and let $i_0\in\{1,2,3\}$. Then there exists a vector $X\in T_{[\rho]} \mathcal{M} (\Gamma,\mathrm{SU}(2))$ such that for $i=1,2,3$
\[d_{[\rho]}f_{c_i} X = \delta _{i_0} ^i .\]
\end{Lemma}
The proof of the lemma ~\ref{Lemma 3.4} uses the Fox calculus, a notion of differential calculus on groups. We refer to the annex in section $6$ for the background on the Fox calculus and its use in the computational proof of the lemma ~\ref{Lemma 3.4}.
We hence start by the proof of the proposition.
\begin{proof}[Proof of the proposition ~\ref{Proposition 3.2}]
The complement of $\mathcal{M}_m(\Gamma,\mathrm{SU}(2))$ is the set :
\[ \bigcup_{(q_0,q_1,q_2,q_3)\in\mathbf{Z}^{4}\backslash \{0\}}  \bigg\{[\rho]\in \mathcal{M} (\Gamma,\mathrm{SU}(2))\mid q_1 \theta_1 (\rho)+q_2\theta_2(\rho)+q_3 \theta_3(\rho) = q_0 \pi \bigg\}.\]
Let $q$ be the vector $(q_1,q_2,q_3)$. If $q_0\neq0$ and $q=0$, the relation $q_1 \theta_1 (\rho)+q_2\theta_2(\rho)+q_3 \theta_3(\rho) = q_0 \pi $ is empty. We then only have the case $q\neq 0$ to consider.
The proposition will be hence proved if for every $(q_0,q_1,q_2,q_3)\in\mathbf{Z}^{4}\backslash \{0\}$ with $q\neq 0_{\mathbf{Z}^3}$, the set : 
\[\big\{[\rho]\in \mathcal{M} (\Gamma,\mathrm{SU}(2))\mid \underbrace{q_1 \theta_1 (\rho)+q_2\theta_2(\rho)+q_3 \theta_3(\rho)}_{= \let\scriptstyle\textstyle
    \substack{ \psi_{m,q}([\rho])}} = q_0 \pi \big\}= \psi_{m,q}^{-1}(q_0 \pi)\]
\noindent has null measure.

We will prove that the map $\psi_{m,q}$ is a submersion on $\mathcal{M} (\Gamma,\mathrm{SU}(2))$. We compute that the differential of $\psi_{m,q}$ at a class $[\rho]$ is :
\[d_{[\rho]}\psi_{m,q}  = \sum_{i=1}^3 \frac{-q_i}{\sin(\theta_i(\rho))}. d_{[\rho]}f_{c_i}. \]

Let $i_0$ the first index for which $q_{i_0}\neq 0$. Then for all $[\rho]\in \mathcal{M} (\Gamma,\mathrm{SU}(2))$, the lemma~\ref{Lemma 3.4} gives a vector $X\in T_{[\rho]} \mathcal{M} (\Gamma,\mathrm{SU}(2))$ such that for $i=1,2,3$, 

\[d_{[\rho]}f_{c_i} X = \delta _{i_0} ^i .\]

Hence, we compute : 
\[ d_{[\rho]}\psi_{m,q} X=\frac{-q_{i_0}}{\sin(\theta_{i_0}(\rho))} \neq 0. \]
We so conclude that the map $\psi_{m,q}$ is a submersion and then that the set: 
\[ \{[\rho]\in \mathcal{M} (\Gamma,\mathrm{SU}(2))\mid q_1 \theta_1 (\rho)+q_2 \theta_2(\rho)+q_3 \theta_3(\rho) = q_0 \pi \}\]
\noindent is a submanifold of the character variety with codimension $1$ and hence has null measure. It follows that $\mathcal{M}_m(\Gamma,\mathrm{SU}(2))$ has full measure as the complement of a countable union of codimension $1$ submanifolds.
\end{proof}

\begin{proof}[Proof of the Lemma~\ref{Lemma 3.4}]

Up to applying an element of the mapping class group, we can assume that  $c_1=a_1$, $c_2=a_1a_2$, $c_3=a_2$ and that $i_0=1$.
 
The point is to construct a smooth path of representations $\rho_t$ such that $\rho_0=\rho$ and such that : 
\[\frac{d}{dt}_{|t=0}\mathrm{tr}(\rho_t(a_1))\neq 0, \frac{d}{dt}_{|t=0}\mathrm{tr}(\rho_t(a_1 a_2))=0, \frac{d}{dt}_{|t=0}\mathrm{tr}(\rho_t(a_2))= 0.\]

Up to changing the representative of $\rho$ in its conjugacy class, we assume that $\rho(a_1)$ is the diagonal matrix $\begin{pmatrix}
e^{i\theta_1} & 0 \\
0 & e^{-i\theta_1}
\end{pmatrix}$.

In order to construct the representation $\rho_t$,  we only need to specify the paths of matrices $\rho_t(a_i)$ and $\rho_t(b_i)$, for $i=1,\dots,g$, such that the equation :
\[\prod_{i=1}^g[\rho_t(a_i),\rho_t(b_i)]=\mathrm{id}\]
\noindent holds.
We let : \[\rho_t(a_1)=\begin{pmatrix}
e^{it} & 0 \\
0 & e^{-it}
\end{pmatrix}\rho(a_1).\] Hence we compute 
\[\frac{d}{dt}_{|t=0}\mathrm{tr}(\rho_t(a_1))=-2\sin(\theta_1)\neq 0\]
because $\rho(a_1)\neq\pm \mathrm{id}$.
\\We impose $\rho_t (a_i)=\rho(a_i)$ and $\rho_t (b_i)=\rho(b_i)$ for $i>2$ (if $g>2$).
To obtain the other conditions on the differentials of traces, we set: 
\[\rho_t(a_2)=g(t)\rho(a_2)g(t)^{-1}, \rho_t(a_1 a_2)=h(t)\rho(a_1)\rho(a_2)h(t)^{-1},\]
\[\rho_t(b_1)=B_1 (t)\text{ and } \rho_t(b_2)=B_2(t)\]
\noindent where $(g(t))_t$, $(h(t))_t$, $B_1 (t)$ and $B_2 (t)$ are smooth paths in $\mathrm{SU}(2)$ such that : 
\[\rho_t(a_1 a_2)=\rho_t(a_1)\rho_t(a_2) \text{ and }[\rho_t(a_1),B_1 (t)][\rho_t(a_2),B_2 (t)] \prod_{i=3}^g[\rho_t(a_i),\rho_t (b_i)]=\mathrm{id}.\]
With these conditions, the map $\rho_t :\{a_1,\dots,b_g\}\rightarrow \mathrm{SU}(2)$ extends to a morphism $\rho_t:\Gamma\rightarrow\mathrm{SU}(2)$ which is unique.
We have to prove the existence of such paths.
Let $K: \mathrm{SU}(2)^2 \times \mathrm{SU}(2)^2\times\mathbf{R}  \rightarrow \mathrm{SU}(2)\times \mathrm{SU}(2) $ defined by $K(g,h,B_1,B_2,t)=$
\[\bigg(h\rho(a_1 a_2)^{-1}h^{-1} \rho_t(a_1) g\rho(a_2)g^{-1}, [\rho_t(a_1), B_1][g\rho(a_2) g^{-1},B_2]\prod_{i=3}^g[\rho(a_i),\rho(b_i)]\bigg).\]
\begin{Claim}
\label{Claim 1}
If $\rho:\Gamma\rightarrow \mathrm{SU}(2)$ has a discrete centralizer, then the map $K$ is a submersion at the point $(\mathrm{id},\mathrm{id}, \rho(b_1),\rho(b_2),0)$.
\end{Claim}

In the appendix, we prove the claim~\ref{Claim 1} and we continue the proof of Lemma~\ref{Lemma 3.4} assuming that it is true.
Hence the preimage $K^{-1}(\mathrm{id}, \mathrm{id})$ is a submanifold of codimension $1$ of $\mathrm{SU}(2)^2 \times \mathrm{SU}(2)^2\times\mathbf{R}$. We so find $\rho_t:\Gamma\rightarrow \mathrm{SU}(2)$ with the conditions we hoped on the traces. It gives, up to multiply by a constant, a vector $X\in T_{[\rho]} \mathcal{M} (\Gamma,\mathrm{SU}(2))$ which verifies the conclusion of the lemma~\ref{Lemma 3.4}.
\end{proof}

\subsection{Proof of the ergodicity}
In order to prove the theorem~\ref{Theorem 1} for $n=2$, we will consider a measurable function which is $\mathrm{Tor}(\Sigma)$-invariant and proved that, up to restrict it to a full measure subset, it is invariant under the action of enough Dehn twists to be $\mathrm{Mod}^+(\Sigma)$-invariant.

Let $F:\mathcal{M} (\Gamma,\mathrm{SU}(2))\rightarrow \mathbf{R}$ be a measurable function and assume that $F$ is $\mathrm{Tor}(\Sigma)$-invariant. Let $x\in\{a_1,\dots,b_g,d_1,\dots,d_{g-1}\}$ be a point of the generating set the mapping class group, see $section 2$, fix a mutlicurve $m_x=x\cup c_2\cup c_3$ in $MC_0(\Sigma)$ and denote $h$ the product $T_x T_{c_2}T_{c_3}$. The set :$$\mathcal{M}_{m_x}(\Gamma,\mathrm{SU}(2))$$ has full measure by the proposition~\ref{Proposition 3.2}. As the orbits $\mathrm{U}_{x}\times\mathrm{U}_{c_2}\times\mathrm{U}_{c_3}.[\rho]$ are tori and then are compact in $\mathcal{M} (\Gamma,\mathrm{SU}(2))$, the theorem~\ref{Theorem 2.1} insures the existence of a Borel cross section, we will denote by $\mathcal{S}$, of $\mathcal{M} (\Gamma,\mathrm{SU}(2))$ for the action of $\mathrm{U}_{x}\times\mathrm{U}_{c_2}\times\mathrm{U}_{c_3}$ we will denote $\mathbf{T}^3$ to simplify the notations.

Since $\mathrm{U}_{x}\times\mathrm{U}_{c_2}\times\mathrm{U}_{c_3}$ and $\mathcal{M} (\Gamma,\mathrm{SU}(2))$ verify the assumption of the theorem~\ref{Theorem 2.2}, then the section $\mathcal{S}$ is bimeasurable to the quotient 

\[\mathcal{M} (\Gamma,\mathrm{SU}(2))\slash \mathbf{T}^3\] 

\noindent and the restriction $\mu$ decompose itself, for all function $f:\mathcal{M} (\Gamma,\mathrm{SU}(2))\rightarrow \mathbf{R}^+$, by the formula : 

\[\int  _{\mathcal{M} (\Gamma,\mathrm{SU}(2))}f d\mu= \int  _{S} \bigg(\int  _{\mathbf{T}^3}f(t.s) d\nu_{\mathbf{T}^3}(t) \bigg) d\nu_{\mathcal{S}}(s)\]
\noindent where $\nu_{\mathbf{T}^3}$ is the Haar measure on the tori $\mathbf{T}^3$ given by theorem~\ref{Theorem 2.2} and $\nu_{\mathcal{S}}$ a measure on $\mathcal{S}$ given by the same theorem.

The function $F$ induces a measurable function $\widetilde{F}:\mathcal{S}\times \mathbf{T}^3\rightarrow\mathbf{R}$
defined by : 

\[\widetilde{F}(s,t)=F(t. s).\]

Fix $[\rho]\in\mathcal{M}_{m_x}(\Gamma,\mathrm{SU}(2))$ and let $\widetilde{F}_{|\{\overline{[\rho]}\}\times  \mathbf{T}^3}: \mathbf{T}^3\rightarrow \mathbf{R}$. Such a function is measurable and invariant by the action of $\langle h\rangle$ since $F$ is. Since $[\rho]\in\mathcal{M}_{m_x}(\Gamma,\mathrm{SU}(2))$, the action of the translation $\langle h\rangle$  on $\mathbf{T}^3$ is ergodic with respect to $\nu_{\mathbf{T}^3}$. It implies that $\widetilde{F}_{|\{\overline{[\rho]}\}\times  \mathbf{T}^3}$ is almost everywhere constant and hence that the restriction $F_{|\mathbf{T}^3.[\rho]}$ is almost everywhere constant. Since the Dehn twist $T_x$ acts as a translation of the torus on $\mathrm{U}_{x}\times\mathrm{U}_{c_2}\times\mathrm{U}_{c_3}.[\rho]$, we deduce that on a full measure subset of $\mathrm{U}_{x}\times\mathrm{U}_{c_2}\times\mathrm{U}_{c_3}.[\rho]$ the function $F$ and $F\circ T_x$ are equal. This fact is true for almost every $[\rho]\in\mathcal{M}_{m_x}(\Gamma,\mathrm{SU}(2))$ which has full measure by~\ref{Proposition 3.2}.
It follows that $F_{|\mathcal{M}_{m_x}(\Gamma,\mathrm{SU}(2))}$ is almost everywhere invariant by the Dehn twist $T_x$.

We then deduce that on the space 
\[\bigcap_{x\in\{a_1,\dots,b_g,d_1,\dots,d_{g-1}\}}\mathcal{M}_{m_x}(\Gamma,\mathrm{SU}(2)),\]

\noindent which has full measure in $\mathcal{M} (\Gamma,\mathrm{SU}(2))$, the function $F$ is almost everywhere invariant by the Dehn twists $T_x$, for all $x\in\{a_1,\dots,b_g,d_1,\dots,d_{g-1}\}$.
It implies that $F$ is almost everywhere invariant under the action of $\mathrm{Mod}^+(\Sigma)$, which is known to be ergodic since the theorem~\ref{Theorem 1.1}. Hence $F$ is almost everywhere constant and this proves the ergodicity of the Torelli group action on $\mathcal{M} (\Gamma,\mathrm{SU}(2))$.

\section{Ergodicity for $G=\mathrm{SU}(n)$}
  
  This part is devoted to the proof of the theorem~\ref{Theorem 1} in the general case. We will use the same strategy than the previous section, adapting the tools. Let us introduce some notations we will use. 
  
   \subsection{Torus actions on $\mathcal{M}^{\alpha\text{-reg}}(\Gamma,\mathrm{SU}(n))$}
  
  For a matrix $A\in\mathrm{SU}(n)$ with distinct eigenvalues, denote $\lambda_1(A),\dots,\lambda_n(A)$ its eigenvalues and $\theta_1(A),\dots,\theta_n(A)$ their arguments we can express by up to the sign: 
  \[\cos^{-1}\bigg(\frac{\lambda_i(A)+\lambda_i(A)^{-1}}{2}\bigg)\] \noindent with the normalisation $0\leq\theta_1(A)\leq\dots\leq\theta_n(A)<2\pi$.
   The group $\mathrm{SU}(n)$ has rank $n-1$ and every maximal torus is conjugated to the group : 
  \[
  \Bigg\{  
 \begin{pmatrix}
    e^{i \theta_1} & 0 & \dots & 0 \\
    0 &  e^{i \theta_2} & \dots & 0 \\
    \vdots & \vdots & \ddots & \vdots \\
    0 & 0 & \dots &  e^{i \theta_n}
  \end{pmatrix},
  (\theta_1,\dots,\theta_n)\in [0,2\pi[ \text{ and }\prod_{k=1}^n e^{i\theta_k}=1
  \Bigg\}
  \]
  
  \noindent which is isomorphic to $\mathbf{T}^{n-1}$.
  
  A matrix $A\in \mathrm{SU}(n)$ is \emph{regular} if its eigenvalues are simple. Similarly, for a curve $\alpha$, a character $[\rho]\in \mathcal{M}(\Gamma,\mathrm{SU}(n))$ is $\alpha$-\emph{regular} if the matrix $\rho(\alpha)$ is regular. We will denote $\mathcal{M}^{\alpha\text{-reg}}(\Gamma,\mathrm{SU}(n))$ the subsets of $\alpha$-regular $\mathrm{SU}(n)$-characters. For all curve $\alpha$, the subset $\mathcal{M}^{\alpha\text{-reg}}(\Gamma,\mathrm{SU}(n))$ is an open subset of $\mathcal{M}(\Gamma,\mathrm{SU}(n))$ and has full measure in it (see subsection $4.2$).

 We will define actions of a $(n-1)$-torus $\mathbf{T}^{n-1}=(\mathbf{S}^1)^{n-1}$ on the character variety $\mathcal{M}(\Gamma,\mathrm{SU}(n))$. Let $z=(z_1,...,z_{n-1})\in \mathbf{T}^{n-1}$ and $h_z$, the associated diagonal matrix : \[\text{diag} \big(z_1,\dots,z_{n-1},\frac{1}{z_1\cdots z_{n-1}}\big)\] in $\mathrm{SU}(n)$.
  
  Let $A\in\mathrm{SU}(n)$ be a regular matrix. There exists a unique decomposition $[e_{1}]\oplus\cdots\oplus[e_n]$ of $\mathbf{C}^n$ in lines such that: 
  \[A e_i =\lambda_i (A) e_i\]
\noindent for all $i\in\{1,\dots,n\}$.
 
 Let $\alpha$ be a simple and closed curve and let $[\rho]\in\mathcal{M}^{\alpha\text{-reg}}(\Gamma,\mathrm{SU}(n))$. With the same notations than the section $3$ and in a basis of $\mathbf{C}^n$ in which $\rho(\alpha)=h_{(\lambda_1(\rho(\alpha)),\cdots\lambda_{n-1}(\rho(\alpha)))}$, we define for $z\in\mathbf{T}^{n-1}$, if $\alpha$ is non-separating, the representation $z\cdot\rho$ by :
\[ 
z\cdot\rho(\gamma)= 
\left\{\begin{array}{l l}
\rho(\gamma) & \text{if }\gamma\in\pi_1(\Sigma |\alpha)\\
h_z \rho(\beta)& \text{if }\gamma=\beta\\
\end{array}\right.
\]
 and if $\alpha$ is separating by :
 \[ 
z\cdot\rho(\gamma)= 
\left\{\begin{array}{l l}
\rho(\gamma) & \text{if }\gamma\in\pi_{1}(\Sigma_1)\\
h_z \rho(\gamma)h_z^{-1}
& \text{if }\gamma\in\pi_1(\Sigma_2)\\
\end{array}\right..
\]

It then defines an action, which depends hence of $\alpha$, of the torus $\mathrm{U}_{\alpha}:=\mathbf{T}^{n-1}$ on the subspace of $\alpha$-regular characters.
We hence remark that the action of Dehn twists along $\alpha$ on
 $\mathcal{M}^{\alpha\text{-reg}}(\Gamma,\mathrm{SU}(n))$ is given by: 
\[T_{\alpha}.[\rho]= \lambda(\rho(\alpha))\cdot[\rho]\]
where $\lambda(\rho(\alpha))=(\lambda_1(\rho(\alpha)),\dots,\lambda_{n-1}(\rho(\alpha)))$.

A direct computation shows the essential fact that if $\alpha$ and $\beta$ are disjoint curves, then the actions of the tori $\mathrm{U}_{\alpha}$ and $\mathrm{U}_{\beta}$ on $\mathcal{M}^{\alpha\text{-reg}}(\Gamma,\mathrm{SU}(n))\cap \mathcal{M}^{\beta\text{-reg}}(\Gamma,\mathrm{SU}(n))$ commute.
It hence defines an action of $\mathrm{U}_{\alpha}\times\mathrm{U}_{\beta}$ on the space of $\alpha$-regular and $\beta$-regular characters.

We so can state an analogue of the lemma ~\ref{Lemma 3.2} :
 
 \begin{Lemma}
 \label{Lemma 4.1}
 Let $[\rho]\in \mathcal{M} (\Gamma,\mathrm{SU}(n))$ and suppose that there exist $c_1,\dots, c_{\ell}$ be pairwise disjoints, simple and closed curves of $\Sigma$ such that:
 \[\theta_1(\rho(c_1)),\dots,\theta_{n-1}(\rho(c_1)),\dots,\theta_1 (\rho(c_{\ell})),\dots,\theta_{n-1} (\rho(c_{\ell})),\pi\]
 are linearly independent over $\mathbf{Q}$.
If we denote $h=T_{c_1} \cdots T_{c_{\ell}}$, then the action of $h$ on the orbit $\mathrm{U}_{c_1}\times\dots\times\mathrm{U}_{c_{\ell}}.[\rho]$ is ergodic with respect to the Lebesgue measure on this torus orbit.
\end{Lemma}

\begin{proof}

As for the proof of the lemma~\ref{Lemma 3.2}, the action of the Dehn twists $T_{c_i}$ is an action by translation on the torus. The flows commute on the character variety because the curves are disjoint and the orbit is given by the formula :
\[h^k.[\rho] = (\lambda_1(\rho(c_1))^k,\dots\lambda_{n-1}(\rho(c_{\ell}))^k,\dots,\lambda_1(\rho(c_{\ell}))^k,\dots\lambda_{n-1}(\rho(c_{\ell}))^k).[\rho].\] 

Hence the condition on the $\theta_i (\rho(c_k))$ implies the expected ergodicity by the lemma~\ref{Lemma 3.3}.
\end{proof}

 \subsection{A full measure set and the ergodicity}
 
 We define in this section a full measure subspace of the character variety with the conditions of the previous lemma and conclude with the ergodicity of the Torelli group action.
 
 \begin{Definition}
Let $m=c_1\cup\dots\cup c_{\ell}$ be a multicurve of simple closed and non-separating curves. A class of representation $[\rho]\in\mathcal{M}(\Gamma,\mathrm{SU}(n))$ verifies the \emph{condition} $(M_m)$ if:
\[\theta_1(\rho(c_1))\dots,\theta_{n-1}(\rho(c_1)),\dots,\theta_1 (\rho(c_{\ell})),\dots\theta_{n-1} (\rho(c_{\ell})),\pi\]
 are linearly independent over $\mathbf{Q}$.
\end{Definition}

\begin{Remark}
\label{Remark 4}
If a class of representation $[\rho]$ verifies the condition $(M_m)$ for some $m=c_1\cup\dots\cup c_{\ell}\in MC(\Sigma)$ then $[\rho]$ is $c_i$-regular for all $i\in\{1,\dots,\ell\}$.

To simplify the notations, for a multicurve $m=c_1\cup\dots\cup c_{\ell}$, we will denote by $\mathcal{M}^{m-\mathrm{reg}}(\Gamma,\mathrm{SU}(n))$ the intersection : 
\[\bigcap_{i=1}^{\ell} \mathcal{M}^{c_i -\mathrm{reg}}(\Gamma,\mathrm{SU}(n))\]
\noindent which has full measure.

\end{Remark}

We define, for $m\in MC_0(\Sigma)$, the set:
\[\mathcal{M}_m (\Gamma,\mathrm{SU}(n))=\big\{[\rho]\in \mathcal{M} (\Gamma,\mathrm{SU}(n))\mid [\rho] \text{ statisfies the condition } (M_m)\big\}.\]

Remark~\ref{Remark 4} assures that $\mathcal{M}_m (\Gamma,\mathrm{SU}(n))$ is contained in $\mathcal{M}^{m-\mathrm{reg}}(\Gamma,\mathrm{SU}(n))$.

\begin{Proposition}
\label{Proposition 4.1}
For all $m\in MC_0(\Sigma)$, the set $\mathcal{M}_m (\Gamma,\mathrm{SU}(n))$ has full measure in the character variety.
\end{Proposition}

We can write its complement as the set:
\[ \bigcup_{q=(q_1 ^1,\dots,q_{n-1}^1,\dots, q_1 ^3,\dots q_{n-1}^3)\in\mathbf{Z}^{3(n-1)}\backslash \{0\},\newline q_0\in \mathbf{Z}}  \bigg\{[\rho]\in \mathcal{M} (\Gamma,\mathrm{SU}(n))\mid \sum_{k=1}^3 \sum_{i=1}^{n-1} q_i ^k \theta_i(\rho(c_k))=q_0 \pi \bigg\}.\]

We will show that each set of the previous union is a codimension $1$ submanifold and hence show that the union has null measure.

Let $\psi_{m,q}$ the function $\mathcal{M}^{m-\mathrm{reg}} (\Gamma,\mathrm{SU}(n))\rightarrow \mathbf{R}$ define by the formula : 
\[\psi_{m,q}([\rho])=\sum_{k=1}^3 \sum_{i=1}^{n-1} q_i ^k \theta_i(\rho(c_k)).\]

\begin{Lemma}
\label{Lemma 4.2}
The function $\psi_{m,q}$ is a submersion.
\end{Lemma}

Denote $\pi_i^{k} : [\rho] \mapsto  \lambda_i(\rho(c_k))+\lambda_i(\rho(c_k))^{-1}$ so that $\theta_i(\rho(c_k))=\arccos (\frac{\pi_i^{k}([\rho])}{2})$ up to the sign. We then compute the differential of $\psi_{m,q}$ : 
\[d_{[\rho]}\psi_{m,q} = \sum_{k=1}^3 \sum_{i=1}^{n-1} -\frac{ q_i ^k}{2\sin(\theta_i(\rho(c_k)))} d_{[\rho]} \pi_i^{k}. \]

\begin{Lemma}
\label{Lemma 4.3}
Let $[\rho]$ be a class of representations in $ \mathcal{M} (\Gamma,\mathrm{SU}(n))$, $m=c_1\cup c_2 \cup c_3$ be a pant such that $\rho(c_i)$ does not have $\pm1$ as eigenvalues and let $i_0\in\{1,\dots,n-1\},k_0\in\{1,2,3\}$. Then there exists a vector $X\in T_{[\rho]} \mathcal{M} (\Gamma,\mathrm{SU}(n))$ such that for $(i,k)\in \{1,\dots,n-1\}\times\{1,2,3\}$,
\[d_{[\rho]}\pi_{i}^{c_k} X = \delta _{i_0} ^i\delta_{k_0}^k .\]
\end{Lemma}

\begin{proof}[Proof of lemma~\ref{Lemma 4.3}]

We are then looking for $\rho_t$ approaching $\rho$ such that for all $(j,k)\neq(i_0,k_0)$: 

\[\frac{d}{dt}_{|t=0}\pi_{j}^{k}(\rho_t)=0\] 
\noindent and 
\[\frac{d}{dt}_{|t=0}\pi_{i_0} ^{k_0}(\rho_t)\neq 0.\]

The strategy to find this family of representations is the same than in $\mathrm{SU}(2)$. Up to applying an element of the mapping class group, we assume $c_1=a_1$, $c_2=a_1 a_2$, $c_3= a_2$ and $k_0=1$. Up to changing the representative of $\rho$ in its conjugacy class, we assume that \[\rho(a_1) =\begin{pmatrix}
    e^{i \theta_1} & 0 & \dots & 0 \\
    0 &  e^{i \theta_2} & \dots & 0 \\
    \vdots & \vdots & \ddots & \vdots \\
    0 & 0 & \dots &  e^{i \theta_n}
  \end{pmatrix}\text{ with } \prod_{k=1}^n e^{i\theta_k}=1.\]
   Multiply $\rho(a_1)$ by the diagonal matrix $h_{z_t}$ with $z_t=(z_{t,i})_{i=1,\dots n-1}$ for $z_{t,i_0}=e^{it}$ and $z_{t,k}=1$ otherwise. 
 As previously, we impose $\rho_t(a_1) = h_{z_t}\rho(a_1)$, $\rho_t(a_i)=\rho(a_i)$ and $\rho_t(b_i)=\rho(b_i)$ for all $i>2$ (if the genus of the surface is greater than $2$). To obtain the other conditions, we let: 
 \[\rho_t(a_2)=g(t)\rho(a_2)g(t)^{-1}\text{ and } \rho_t(a_1a_2)=h(t)\rho(a_1)\rho(a_2)h(t)^{-1}\]
 \[\rho_t(b_1)=B_1(t)\text{ and }\rho_t(b_2)=B_(t)\]
 \noindent for paths $g(t),h(t),B_1(t),B_2(t)\in \mathrm{SU}(n)$ which verify : 
  \[\rho_t(a_1 a_2)=\rho_t(a_1)\rho_t(a_2) \text{ and } \prod_{i=1}^g[\rho_t(a_i),\rho_t (b_i)]=\mathrm{id}.\]
  With such conditions, the map $\rho_t:\{a_1,\dots,b_g\}\rightarrow \mathrm{SU}(n)$ extends to a unique morphism $\rho_t:\Gamma\rightarrow \mathrm{SU}(n)$.
  In particular, if such paths exist, we compute : 
  \[\frac{d}{dt}_{|t=0}\pi_{i_0} ^{1}(\rho_t)=-2\sin(\theta_{i_0}(\rho(a_1)))\neq 0\]
  since $\rho(a_1)$ does not have $\pm 1$ as eigenvalue. For such paths the conditions $\frac{d}{dt}_{|t=0}\pi_{j}^{k}(\rho_t)=0$ are clearly verified since we conjugate $\rho(a_2),\rho(a_1a_2)$ by matrices, which does not change the eigenvalues.
  
  To find such paths, we let the map $K: \mathrm{SU}(n)^2 \times \mathrm{SU}(n)^2\times\mathbf{R}  \rightarrow \mathrm{SU}(n)\times \mathrm{SU}(n)$ defined by $K(g,h,B_1,B_2,t)=$

\[\bigg(h\rho(a_1 a_2)^{-1}h^{-1} \rho_t(a_1) g\rho(a_2)g^{-1}, [\rho_t(a_1), B_1][g\rho(a_2) g^{-1},B_2]\prod_{i=3}^g[\rho(a_i),\rho(b_i)]\bigg).\]
  
\begin{Claim}
\label{Claim 2}
If $\rho:\Gamma\rightarrow \mathrm{SU}(n)$ has a discrete centralizer, then the map $K$ is a submersion at the point $(\mathrm{id},\mathrm{id}, \rho(b_1),\rho(b_2),0)$.
\end{Claim}  

Assuming the claim~\ref{Claim 2} which is proved in appendix in section $6$, we hence conclude the existence of the path $\rho_t$ and then to the existence of a vector $X\in T_{[\rho]}\mathcal{M} (\Gamma,\mathrm{SU}(n))$ which verifies :
\[d_{[\rho]} \pi_i^{k} X = \delta_{i_0}^i\delta _{k_0}^k.\]\end{proof}

We now prove the Lemma~\ref{Lemma 4.2}. Let $i_0\in\{1,\dots n-1\}$ and $k_0\in\{1,2,3\}$ the first index such that $q_{i_0}^{k_0}$ is non-zero. Let $X$ be te vector field associated to the index $i_0$ and $k_0$ in the lemma~\ref{Lemma 4.3}.
We then have that \[d_{[\rho]}\psi_{m,q}X =-\frac{ q_{i_0} ^{k_0}}{2\sin(\theta_i(\rho(c_k)))} d_{[\rho]} \pi_{i_0}^{k_0}X\neq 0 \]
\noindent which proves that $\psi_{m,q}$ is a submersion. It proves the lemma~\ref{Lemma 4.2} and then the proposition~\ref{Proposition 4.1}.\\

The proof of the ergodicity use the same arguments than the case of $\mathrm{SU}(2)$. Let $F:\mathcal{M} (\Gamma,\mathrm{SU}(n))\rightarrow \mathbf{R}$ be a measurable and  $\mathrm{Tor}(\Sigma)$-invariant function. For each curve $x$ in the set $\{a_1,\dots,b_g,d_1,\dots,d_{g-1}\}$, we fix a multicurve $m_x=x\cup c_2\cup c_3$.

Replacing the torus $\mathbf{T}^3$ we used for the case $\mathrm{SU}(2)$ by the torus $\mathbf{T}^{3(n-1)}$, we conclude by the same methods that on the space 

\[\bigcap_{x\in\{a_1,\dots,b_g,d_1,\dots,d_{g-1}\}}\mathcal{M}_{m_x}(\Gamma,\mathrm{SU}(n)),\]

which has full measure the proposition~\ref{Proposition 4.1}, the function $F$ is almost everywhere invariant by the Dehn twists $T_x$ for all $x\in\{a_1,\dots,b_g,d_1,\dots,d_{g-1}\}$. It implies that it is almost everywhere invariant by the mapping class group. The theorem~\ref{Theorem 1.2} shows that $F$ is constant on a full measure subset of $ \mathcal{M} (\Gamma,\mathrm{SU}(n))$. This proves the ergodicity of the Torelli group action on $ \mathcal{M} (\Gamma,\mathrm{SU}(n))$.

 \section{Ergodicity for the general cases of semi-simple, connected and compact Lie groups} 
 
 In this section we generalize the proofs of the ergodicity of the Torelli group on character varieties with values in a semi-simple, connected and compact Lie group. We will use the same strategy than the compact Lie groups $\mathrm{SU}(n)$ but need to replace the tools we used by their appropriate analogues in a more general case.
 
 \subsection{Preliminaries on compact Lie group theory}
 
 Let $G$ be a semi-simple, connected and compact Lie group with Lie algebra $\mathfrak{g}$. A maximal torus is a connected and abelian subgroup of $G$ which is maximal for these properties. Such a subgroup exists and fix $T< G$ be a maximal torus. Let $\mathfrak{t}$ its Lie algebra. It is an abelian subalgebra of $\mathfrak{g}$.  The subgroup $T$ is isomorphic to a $r$-dimensional torus $\mathbf{T}^r$ and its Lie algebra $\mathfrak{t}$ is isomorphic to the commutative Lie algebra $\mathbf{R}^r$. We will so use the existence of coordinates on $\mathfrak{t}$ via this isomorphism. Precisely, for $i\in\{1,\dots,r\}$ and $t\in T$, we denote by $\lambda_i(t)$ the projection on the $i$-th factor of $t\in T  \cong\mathbf{T}^r$. 
 
 It is well known that every element of $G$ is contained in a maximal torus. We have however the more precise result (see \cite{zbMATH01060968} for more details) : 
 
 \begin{Theorem2}
Every $k\in G$ is conjugated to an element of $T$. Moreover, all the maximal tori are conjugated and hence are isomorphic to $\mathbf{T}^r$.
  \end{Theorem2}
  
  Remark that a maximal torus can contain two conjugated elements. The integer $r$ is called the \emph{rank} of the group $G$. The \emph{Weyl group} associated to $T$ is the group $\mathrm{N}_{G}(T)\slash T$, where the subgroup $\mathrm{N}_{G}(T)$ is the normalizer of $T$ in $G$.
  
  \begin{Proposition}(\cite{zbMATH03758599})
  The Weyl group is finite.
  \end{Proposition}
  
A weight of $T$ is a real and irreducible representation. Let $\omega$ be a weight of $T$ and $\sigma : G\rightarrow \mathrm{Aut}(V)$ be a representation. The sum of all invariant subspaces of $\sigma|_{T}$ isomorphic to $\omega$ is called the \emph{weight space} associated to $\omega$ of $\sigma$.
Define, for $\mathbf{n}=(n_1,\dots,n_r)\in\mathbf{Z}^r$, the linear form : 

\[ 
\begin{array}{ccccc}

\Theta_{\mathbf{n}}^* &:&   \mathfrak{t} & \to & \mathbf{R}  \\
	 && (x_1,\dots,x_r) & \mapsto    & n_1 x_1+\cdots+ n_r x_r \\ 
\end{array}
\]
\noindent where we use the coordinates on $\mathfrak{t}$ given by the isomorphism $\mathfrak{t}\cong\mathbf{R}^r$ coming from \[T\cong\big(\mathbf{R}\slash\mathbf{Z}\big)^r.\]

 It is then well known that the weights of $T$ are either the trivial one-dimensional representation or the representations $\Theta_{\mathbf{n}}: T\cong\mathbf{T}^r\rightarrow \mathrm{SO}_2 (\mathbf{R})$, for $\mathbf{n}=(n_1,\dots,n_r)\in\mathbf{Z}^r \backslash\{0\}$,  defined by :

\[\Theta_{\mathbf{n}}([x_1,\dots,x_r])=\begin{pmatrix}
\cos(2\pi \Theta_{\mathbf{n}}^* (x_1,\dots,x_r)) & -\sin(2\pi \Theta_{\mathbf{n}}^* (x_1,\dots,x_r)) \\
\sin(2\pi \Theta_{\mathbf{n}}^* (x_1,\dots,x_r)) & \cos(2\pi \Theta_{\mathbf{n}}^* (x_1,\dots,x_r))
\end{pmatrix}.\] 

\begin{Definition}
A linear form $\alpha\in\mathfrak{t}^*$ is a \emph{root} of $G$ if there exists $\mathbf{n}=(n_1,\dots,n_r)\in\mathbf{Z}^r \backslash\{0\}$ such that $\alpha =\Theta_{\mathbf{n}}^* $ and that the weight space of the adjoint representation $\operatorname{Ad}:G\rightarrow \mathrm{GL}(\mathfrak{g})$ associated to $\Theta_{\mathbf{n}}$ is non-trivial.
\end{Definition}

We denote by $\Delta$ the set of roots of $G$.

Since $G$ is a semi-simple Lie group, the Killing form $\langle.,.\rangle$ is a scalar product and the subspace $\mathfrak{t}< \mathfrak{g}$ becomes a Euclidean space. Using the induced isomorphism $\mathfrak{t}\cong\mathfrak{t}^*$, we can see $\Delta$ as a subset of $\mathfrak{t}$ and for $\alpha\in\Delta$, we define the reflection :
\[ r_{\alpha}:\beta\mapsto \beta - \frac{2\langle\beta,\alpha\rangle}{\langle\alpha,\alpha\rangle} \alpha.\]
It is well known that the Weyl group associated to $T$ is isomorphic to the subgroupof $\mathrm{GL}(\mathfrak{t})$:
\[\langle r_{\alpha} |\alpha \in\Delta \rangle.\]
The \emph{alcoves} of $\mathfrak{t}$ are the connected components of \[\mathfrak{t}\backslash \cup_{\alpha\in\Delta, n\in\mathbf{N}} \ker(r_\alpha- n \mathrm{id}).\] 
The Weyl group acts simply transitively on the images in $T$ of alcoves.

\begin{Definition}
 An element $k\in G$ is \emph{regular} if it is contained in a unique maximal torus.
 \end{Definition}
 
Let $M$ be the image by the exponential map $\mathfrak{t}\rightarrow T$ of an alcove of $\mathfrak{t}$, we will say such a $M$ is an alcove of $T$ and let $k\in G$ be a regular element. There exists a unique class $\overline{g_k}$ of $ G\slash \mathrm{Z}_{G} (k)$, with $\mathrm{Z}_{G} (k)$ the centralizer of $k$ in $G$, such that :
 \[g_k k g_k ^{-1}\in M.\]

\begin{Example}
For $G=\mathrm{SU}(n)$ and $T$ the set of diagonal matrices, the roots are given by $\lambda_i-\lambda_k$ for $i,k\in\{1,\dots,n\}$ such that $i\neq k$ and where the $\lambda_i$ are the eigenvalues. The Weyl group is then the symmetric group $\mathfrak{S}_{n}$. \end{Example}

 \subsection{Density of some orbits}
 
  Let $G$ be a semi-simple, connected and compact Lie group of rank $r$, let $T$ be a maximal torus and $M$ be an alcove of $T$.
 
 Let $\alpha\in\Gamma$ be a simple curve. A character $[\rho]\in\mathcal{M}(\Gamma,G)$ is $\alpha$-\emph{regular} if $\rho(\alpha)$ is regular. Then there exists a unique class $\overline{g_{\rho(\alpha)}}$ in the quotient 
 $G\slash \mathrm{Z}_{G} (\rho(\alpha))$, such that :
 \[g_{\rho(\alpha)}\rho(\alpha)g_{\rho(\alpha)}^{-1}\in M.\]
 The set of $\alpha$-regular characters is an open subset of $\mathcal{M}(\Gamma,G)$ and has full measure. The maximal torus $T$ acts on the space $\mathrm{Hom}^{\alpha-\text{reg}}(\Gamma,G)$ of $\alpha$-regular representations via the action, defined if $\alpha$ is non-separating, for $t\in T  $ and $\rho\in\mathrm{Hom}^{\alpha-\text{reg}}(\Gamma,G)$, by: 
 \[ 
t\cdot\rho(\gamma)= 
\left\{\begin{array}{l l}
g_{\rho(\alpha)}\rho(\gamma)g_{\rho(\alpha)}^{-1} & \text{if }\gamma\in\pi_1(\Sigma |\alpha)\\
t g_{\rho(\alpha)}\rho(\beta)g_{\rho(\alpha)}^{-1}& \text{if }\gamma=\beta\\
\end{array}\right.
\]
and if $\alpha$ is separating by :
 \[ t\cdot\rho(\gamma)= 
\left\{\begin{array}{l l}
g_{\rho(\alpha)}\rho(\gamma)g_{\rho(\alpha)}^{-1} & \text{if }\gamma\in\pi_{1}(\Sigma_1)\\
t g_{\rho(\alpha)}\rho(\gamma)g_{\rho(\alpha)}^{-1}t^{-1}
& \text{if }\gamma\in\pi_1(\Sigma_2).\\
\end{array}\right.
\]
\noindent we use the same notations than the sections $3$ and $4$.
Since this action commutes with the conjugation action of $G$ on the representations, we defined an action, which only depends of the curve $\alpha$, of the maximal torus $\mathrm{U}_{\alpha} := T$ on the $\alpha$-regular characters.
 As for the case of $\mathrm{SU}(n)$, we have an expression of the Dehn twist $T_{\alpha}$, along $\alpha$, action on the subspace of $\alpha$-regular characters via the formula :
 \[T_{\alpha}.[\rho]= (g_{\rho(\alpha)}\rho(\alpha)g_{\rho(\alpha)}^{-1}).[\rho].\]

 We then precise the important fact that for two disjoint curves $\alpha$ and $\beta$, the actions of the maximal tori $\mathrm{U}_{\alpha}$ and $\mathrm{U}_{\beta}$ on $\mathcal{M}^{\alpha-\mathrm{reg}}(\Gamma,G)\cap \mathcal{M}^{\beta-\mathrm{reg}}(\Gamma,G)$ commute. It hence implies an action of the product $\mathrm{U}_{\alpha}\times \mathrm{U}_{\beta}$ on the previous intersection. We define the map $t_{\alpha}: \text{Hom}^{\alpha-\text{reg}}(\Gamma,G)\rightarrow M$ by : 
 \[t_{\alpha}(\rho)=g_{\rho(\alpha)}\rho(\alpha)g_{\rho(\alpha)}^{-1}.\]

For every $i\in\{1,\dots,r\}$,the projection $\lambda_i$ induces a function, we denote by $\lambda_{i,\alpha}$, on the space $\text{Hom}^{\alpha-\text{reg}}(\Gamma,G)$.
 
 If we conjugate $\rho$ by $g\in G$, we obtain by uniqueness of $g_{\rho(\alpha)}$ up to the centralizer of $\rho(\alpha)$, that there exists $z\in\mathrm{Z}_{G}(\rho(\alpha))$ such that : 
 \[g_{g\rho(\alpha)g^{-1}}g=g_{\rho(\alpha)}z\]
\noindent and hence we obtain that $t_{\alpha}$ is invariant under the conjugation action of $G$ on the representation variety $\text{Hom}^{\alpha-\text{reg}}(\Gamma,G)$ and descends to a map \[\mathcal{M}^{\alpha-\text{reg}}(\Gamma,G)\rightarrow M,\] \noindent we will denote $t_{\alpha}$ again.
 
 An element $k\in G$ is said \emph{generic} if $\langle g_k k g_k ^{-1} \rangle$ is dense in $T$.  
 \begin{Claim} An element $k\in G$ is generic if and only if for all non-trivial character $\chi:T\rightarrow \mathbf{S}^1$, 
 \[\chi(g_k k g_k ^{-1})\neq 1.\]
 \end{Claim}
 \begin{proof}
 Let $\phi:T\rightarrow \mathbf{T}^r$ be the isomorphism we mentioned. Then the map $\chi(\phi^{-1})$ is a non-trivial character of the torus $\mathbf{T}^r$. Since a character of $\mathbf{T}^r$ is induced by a linear form of $\mathbf{R}^r$ which preserved $2\pi\mathbf{Z}^r$, it as the form $(x_1,\dots,x_r)\mapsto n_1 x_1 +\cdots +n_r x_r$ with the $n_i\in\mathbf{Z}$ are not all zero. An element $k\in T$ is generic if and only if $\phi(k)$ is generic, that mean if it generates a dense subgroup in $\mathbf{T}^r$. We then have that $\phi(k)$ is generic if and only if $\chi(\phi^{-1})(\phi(k))\neq 1$.
 \end{proof}
 
 In particular, a generic element of $k$ is regular.
 
 \begin{Lemma}
Let $\rho:\Gamma\rightarrow G$ be a representation and $\alpha$ be a simple close curve  such that $\rho(\alpha)$ is generic. Then the orbit 
 \[\langle T_{\alpha}\rangle\cdot[\rho]\]
 is dense in the torus orbit $\mathrm{U}_{\alpha}.[\rho]$. Moreover the action of $\langle T_{\alpha}\rangle$ on $\mathrm{U}_{\alpha}.[\rho]$ is ergodic with respect to the Lebesgue measure.
 \end{Lemma}
 
  \begin{Definition}
 Let $m=c_1\cup\dots\cup c_{\ell}$ be a multicurve. A class of representation $[\rho]\in\mathcal{M}(\Gamma,G)$ is $m$-\emph{regular} if the elements $\rho(c_1),\dots,\rho(c_{\ell})$ are regular and we denote by $\mathcal{M}^{m-\mathrm{reg}}(\Gamma,G)$ the set of $m$-regular characters.
 Precisely we have : 
 \[\mathcal{M}^{m-\mathrm{reg}}(\Gamma,G)=\bigcap_{i=1}^{\ell}\mathcal{M}^{c_i-\mathrm{reg}}(\Gamma,G).\]
 
 Since the curves $c_1,\dots, c_{\ell}$ are disjoint, the action of the tori $\mathrm{U}_{c_1},\dots,\mathrm{U}_{c_{\ell}}$ on  
  commute and the orbits are the torus orbits $\mathrm{U}_{c_1}\times\dots\times\mathrm{U}_{c_{\ell}}.[\rho]$.
 \end{Definition}
 
In this general setting and similarly to the lemmas~\ref{Lemma 3.2} and~\ref{Lemma 4.1}, we state : 
 
 \begin{Lemma}
 \label{Lemma 5.2}
 Let $[\rho]\in \mathcal{M}(\Gamma,G)$ and suppose that there exist $c_1,\dots, c_{\ell}$ be pairwise disjoint, simple and closed curves of $\Sigma$ such that for all non-trivial character $\chi : T^{\ell} \rightarrow \mathbf{S}^1$ : 
 
 \[\chi(t_{c_1}([\rho]),\dots,t_{c_{\ell}}([\rho]))\neq 1\]
 
 Then, if we denote $h=T_{c_1} \cdots T_{c_{\ell}}$, then the action of $\langle h\rangle$ on $\mathrm{U}_{c_1}\times\dots\times\mathrm{U}_{c_{\ell}}.[\rho]$ is ergodic with respect to the Lebesgue measure.

 \end{Lemma}
 
 To simplify the notations we denote by $T^{\ell}$ the product $\mathrm{U}_{c_1}\times\dots\times\mathrm{U}_{c_{\ell}}$.
 
 \begin{proof}[Proof of the lemma~\ref{Lemma 5.2}]
 
 Let $\phi$ be the isomorphism $T\cong \mathbf{T}^r$ and let $\chi : T^{\ell} \rightarrow \mathbf{S}^1$ be a non-trivial character. Then the composition $\chi\circ(\phi^{-1},\dots,\phi^{-1})$ is a non-trivial character of $\mathbf{T}^{r\ell}$, we identify with $\mathbf{R}^{r\ell}\slash\mathbf{Z}^{r\ell}$. The action of $h$ is then given by the translation of vector
 \[(\theta_1(\rho(c_1)),\dots, \theta_r(\rho(c_1)),\dots,\theta_1(\rho(c_{\ell})),\dots,\theta_r(\rho(c_{\ell})))\]
 \noindent where $\theta_i(\rho(c_k))$ is the argument of $\lambda_i (\rho(c_k))$. Then, by the lemma~\ref{Lemma 3.3}, the action of $h$ is ergodic on the torus orbit $\mathrm{U}_{c_1}\times\dots\times\mathrm{U}_{c_{\ell}}.[\rho]$ with respect to the Lebesgue measure if and only if \[\theta_1(\rho(c_1)),\dots, \theta_r(\rho(c_1)),\dots,\theta_1(\rho(c_{\ell})),\dots,\theta_r(\rho(c_{\ell}))\text{ and }1\] \noindent are linearly independent over $\mathbf{Q}$. As a character of a torus is given by a linear form of $\mathbf{R}^{r\ell}$ with integer coefficient, this condition is equivalent that for all non-trivial character $\chi '$ of $\mathbf{T}^{r\ell}$, 
 \[\chi' (\theta_1(\rho(c_1)),\dots, \theta_r(\rho(c_1)),\dots,\theta_1(\rho(c_{\ell})),\dots,\theta_r(\rho(c_{\ell})))\neq 1.\]
 The hypothesis allows then to conclude.
 \end{proof}
 
  \subsection{Proof of the ergodicity}
 
 We will adapt the previous proofs of ergodicity of sections $3$ and $4$ with the condition of the lemma~\ref{Lemma 5.2}.
 
 \begin{Definition}

A class of representation $[\rho]\in\mathcal{M}(\Gamma,G)$ verifies the \emph{condition} $(M_m)$ if for all non-trivial character $\chi : T^{\ell} \rightarrow \mathbf{S}^1$, 

 \[\chi(t_{c_1}([\rho]),\dots,t_{c_{\ell}}([\rho]))\neq 1.\]
 
\end{Definition}

We introduce : 
\[\mathcal{M}_m(\Gamma,G)=\big\{[\rho]\in \mathcal{M} (\Gamma,G)\mid [\rho] \text{ satisfies the condition } (M_m)\big\}.\]

\begin{Remark}
The set $\mathcal{M}_m(\Gamma,G)$ is contained in $\mathcal{M}^{m-\mathrm{reg}}(\Gamma,G)$.
\end{Remark}

We hence prove the following :

\begin{Proposition}
\label{Proposition 5.3}
For all $m\in MC_0(\Sigma)$, the space $\mathcal{M}_m(\Gamma,G)$ has full measure in the character variety.
\end{Proposition}

As in the previous cases, we prove that the set we introduced is the complement of a countable union of submanifold of codimension $1$.

The strategy we use is the same than the propositions~\ref{Proposition 3.2} and~\ref{Proposition 4.1}. Let $m=c_1 \cup c_2 \cup c_3 \in MC_0(\Sigma)$, we write the complement of the set of characters which verify the condition $(M_m)$ by the union :

\[\bigcup_{\underset{\text{non-trivial character}}{\chi : T^3 \rightarrow \mathbf{S}^1}} \bigg\{[\rho]\in \mathcal{M} (\Gamma,G)\mid \chi(t_{c_1}([\rho]), t_{c_2}([\rho]),t_{c_3}([\rho]))= 1\bigg\}. \]

We will hence prove that for all non-trivial character $\chi : T^3 \rightarrow \mathbf{S}^1$, the set 
\[\bigg\{[\rho]\in \mathcal{M} (\Gamma,G)\mid \chi(t_{c_1}([\rho]), t_{c_2}([\rho]),t_{c_3}([\rho]))= 1\bigg\}\] has null measure as a preimage of $1$ by the map $\psi_{\chi, m}=\chi(t_{c_1}(\cdot), t_{c_2}(\cdot),t_{c_3}(\cdot))$, defined on $\mathcal{M}^{m-\mathrm{reg}}(\Gamma,G)$, we will prove to be a submersion. It is the goal of the following lemma. 

\begin{Lemma}
\label{Lemma 5.3}
For all non-trivial character $\chi : T^3 \rightarrow \mathbf{S}^1$, the map $\psi_{\chi, m}$ is a submersion.
\end{Lemma}

\begin{proof}
It suffices, for $[\rho]\in\mathcal{M} (\Gamma,G)$, to find a vector $X\in T_{[\rho]}\mathcal{M} (\Gamma,G)$ such that :
\[d_{[\rho]}\psi_{\chi,m}X \neq 0.\]

Write : 
\[d\psi_{\chi,m}X=d\chi  (dt_{c_1}X, dt_{c_2}X,dt_{c_3}X)\]

\noindent where, for $\phi$ be the isomorphism $T\cong\mathbf{T}^r$ we use in the proof of the lemma~\ref{Lemma 5.2} and $d_{[\rho]}t_{c_k}X$ be :
\[ d\phi^{-1}_{\phi(t_{c_k}([\rho]))}\bigg(\frac{d}{dt}_{|t=0}\lambda_{1}(\rho_t (c_{k})),\dots,\frac{d}{dt}_{|t=0}\lambda_{r}(\rho_t (c_{k}))\bigg),\]

\noindent where $(\rho_t)_t$ is the path tangent to $X$.

Let $k_0\in\{1,2,3\}$ and $i_0\in\{1,\dots,r\}$ such that $d\chi e_{i_0}\neq 0$, with $(e_i)_i$ denote the canonical basis of the  $k_0$-th copy of $\mathbf{R}^r$. Such integers exist because $\chi$ is not the trivial character.

We are looking for $\rho_t$ approaching $\rho$ such that for all $(j,k)\neq(i_0,k_0)$ : 
\[\frac{d}{dt}_{|t=0}\lambda_{j}(\rho_t (c_{k}))=0\text{ and }\frac{d}{dt}_{|t=0}\lambda_{i_0}(\rho_t (c_{k_0}))\neq 0.\] 

We assume that $c_1=a_1$, $c_2=a_1 a_2,c_3= a_2$ and $k_0=1$. We multiply $\rho(a_1)$ by the element $u_t \in M$ which corresponds, by $T\cong\mathbf{T}^r$, to the vector $(1,\dots, e^{it},\dots,1)$ of $\mathbf{T}^r$ (with $e^{it}$ in $i_0$-th position) and we impose  $\rho_t(a_i)=\rho(a_i)$ and $\rho_t(b_i)=\rho(b_i)$ when $i>2$ (in genus $g>2$) and define $\rho_t(a_1)=u_t\rho(a_1)$, $\rho_t(a_2)=g(t)\rho(a_2) g(t)^{-1}$, $\rho_t(b_1)=B_1(t)$, $\rho_t(b_2)=B_2(t)$ and $\rho_t(a_1 a_2)=h(t)\rho(a_1 a_2)h(t)^{-1}$ for smooth paths $g(t),h(t),B_1(t),B_2(t)\in G$ such that : 
\[\rho_t(a_1 a_2)=\rho_t(a_1)\rho_t(a_2) \text{ and } \prod_{i=1}^g [\rho_t(a_i),\rho_t(b_i)]=1.\]

Let the map $K: G^2 \times G^2\times\mathbf{R}  \rightarrow G\times G$ defined by $K(g,h,B_1,B_2,t)=$

\[\bigg(h\rho(a_1 a_2)^{-1}h^{-1} \rho_t(a_1) g\rho(a_2)g^{-1}, [\rho_t(a_1), B_1][g\rho(a_2) g^{-1},B_2]\prod_{i=3}^g[\rho(a_i),\rho(b_i)]\bigg).\]

\begin{Claim}
\label{Claim 4}
If $\rho:\Gamma\rightarrow G$ has a discrete centralizer, then the map $K$ is a submersion at the point $(\mathrm{id},\mathrm{id}, \rho(b_1),\rho(b_2),0)$.
\end{Claim}

The claim~\ref{Claim 4} allows to find a path $(\rho_t)_t$ of representations and then a vector $X\in T_{[\rho]}\mathcal{M} (\Gamma,G)$ such that :
\[d_{[\rho]}\psi_{\chi,m}X \neq 0.\qedhere\] \end{proof}

The set of characters $ T^3 \rightarrow \mathbf{S}^1$ is countable because such a character is given by a linear form : 

\[\tilde{\chi} : \mathbf{R}^{3r}\rightarrow \mathbf{R},\]

\noindent such that $\tilde{\chi}(\mathbf{Z}^{3r})\subset \mathbf{Z}$. Hence there is a countable number of possibilities to obtain characters of the $3r$-torus, looking the image by $\tilde{\chi}$ of the canonical basis.

Since the complement of $\mathcal{M}_m(\Gamma,G)$ is a countable union of codimension 1 submanifolds and hence a countable union of null measure sets, we conclude to the proposition~\ref{Proposition 5.3}. 

 We conclude the proof of the ergodicity with the same arguments than the previous cases. We prove then that all $\mathrm{Tor}(\Sigma)$-invariant and measurable function $\mathcal{M}(\Gamma,G)\rightarrow \mathbf{R}$ can be restricted to a full measure set on which it will be invariant under the generators of the mapping class group and then the theorem~\ref{Theorem 1.2} allows to conclude that such a function is almost everywhere constant, that is the Torelli group action on $\mathcal{M}(\Gamma,G)$ is ergodic.  
\appendix
 \section{Appendix : Fox calculus and proof of the main claims}
 
 \subsection{Fox Calculus}
  
 This appendix consists in giving the tools to do differential calculus on words of free groups. A derivation of a finitely-generated free group $F=\langle s_1,\dots,s_r\rangle$ is a $\mathbf{Z}$-linear map $D: \mathbf{Z}F\rightarrow  \mathbf{Z}F$ where $ \mathbf{Z}F$ is the group ring of $F$, verifying the cocycle relation: 

\[D(xy)=D(x)+xD(y).\] 

R.H.Fox proved that : 

\begin{Theorem2}(\cite{zbMATH03229404})
The set of derivations of $F$ is generated as $\mathbf{Z}[F]$-module by elements $\frac{\partial}{\partial x_i}$, for $i=1,\dots,r$, such that : 

\[\frac{\partial}{\partial x_i}(x_j)=\delta_i ^j.\]
\end{Theorem2}

This theorem and the formula $D(xy)=D(x)+xD(y)$ allow to compute, for all $w\in F$, the derivatives $\frac{\partial}{\partial x_i}(w)$. 

\begin{Example}
\label{Example 2}
[\cite{zbMATH03916723}, section $3$] If $F$ is the free group $\langle a,b \rangle$, then we compute :

\[\frac{\partial}{\partial a}(a^{-1})=-a^{-1},\frac{\partial}{\partial a}(ab)=1,\frac{\partial}{\partial b}(ab)=a\]
 \[\frac{\partial}{\partial a}(aba^{-1})=1-aba^{-1},\frac{\partial}{\partial b}(aba^{-1})=a\]
\[ \frac{\partial}{\partial a}(aba^{-1}b^{-1})=1-aba^{-1} \text{ and }  \frac{\partial}{\partial b}(aba^{-1}b^{-1})=a-aba^{-1}b^{-1}.
 \]

\end{Example}

We will use the Fox calculus to complete the proof of the lemma~\ref{Lemma 3.4}. Indeed this theory allows to derivate words in term of generators. For $i=1,\dots,r$, denote by $dx_i$ the projection \[\mathfrak{g}^r\rightarrow \mathfrak{g}.\] These projections are the derivatives of the projection maps $x_i:G^r\rightarrow G$. Let $w\in F$ be a word. It defines a map $G^r\rightarrow G$ whose derivative is given by : 

\[Dw = \sum_{i=1}^r \frac{\partial w}{\partial x_i}dx_i.\]

Then the computations of the $\frac{\partial w}{\partial x_i}$ are similarly than the computations the example~\ref{Example 2} allows to derivate the word function on $G$.
\subsection{Proof of the main claims}
Let $[\rho]\in\mathcal{M}(\Gamma,G)$ and let $(\rho_t(a_1))_t$ be the paths of elements of $G$ we constructed in section $5$.  We have to prove that the map $K: G^4\times \mathbf{R} \rightarrow G\times G$ defined by $K(g,h,B_1,B_2,t)=$

\[\bigg(h\rho(a_1 a_2)^{-1}h^{-1} \rho_t(a_1)g\rho(a_2) g^{-1}, [\rho_t(a_1), B_1][g\rho(a_2) g^{-1},B_2]\prod_{i=3}^g [\rho(a_i),\rho(b_i)] \bigg),\]

\noindent is a submersion at the point $(\mathrm{id}, \mathrm{id}, \rho(b_1),\rho(b_2),0)$.
To simplify the notations, let $K_1$ and $K_2$ the first and the second coordinates of $K$.

Its differential at the point $(\mathrm{id}, \mathrm{id}, B_1,B_2,0)$ has the form :

\[\begin{pmatrix}
D_{(\mathrm{id}, \mathrm{id}, \rho(b_1),\rho(b_2))}K_{1,0}  \\
  D_{(\mathrm{id}, \mathrm{id}, \rho(b_1),\rho(b_2))}K_{2,0}  \\

\end{pmatrix}
\]

\noindent where $D_{(\mathrm{id}, \mathrm{id}, \rho(b_1),\rho(b_2))}K_{i,0}: \frak{g}^4 \rightarrow \frak{g}$ is the tangent map of $K_{i,0}(.,.,.,.)=K_i(.,.,.,.,0) :G^4 \rightarrow G$ at the point $(\mathrm{id}, \mathrm{id}, \rho(b_1),\rho(b_2))$.

We hence compute that : 

\[D_{g,h,B_1,B_2}K_{1,0} = \operatorname{Ad}\big(\frac{\partial}{\partial g}(K_{1,0})\big)dg + \operatorname{Ad}\big(\frac{\partial}{\partial h}(K_{1,0})\big)dh \]

\noindent with \[\frac{\partial}{\partial g}(K_{1,0})(g,h)=h\rho(a_1 a_2)^{-1}h^{-1} \rho(a_1) (\mathrm{id}-g\rho(a_2) g^{-1})\] and \[\frac{\partial}{\partial h}(K_{1,0})(g,h)=\mathrm{id}-h\rho(a_1 a_2)^{-1}h^{-1}.\] Since for two orthogonal transformations $T$ and $S$ of a Euclidean vector space, the orthogonal subspace to $\mathrm{Im}\big(S(\mathrm{id}-T)\big)$ is the kernel of $\mathrm{id}- STS^{-1}$, we have that the space
\[\operatorname{Ad}\big(\frac{\partial}{\partial g}(K_{1,0})(\mathrm{id},\mathrm{id})\big)(\frak{g}^4)^{\perp}\]
\noindent is the kernel of the operator $\operatorname{Ad}(\mathrm{id}-\rho(a_2)^{-1})$. Similarly the subspace 
\[\operatorname{Ad}\big(\frac{\partial}{\partial h}(K_{1,0})(\mathrm{id},\mathrm{id})\big)(\frak{g}^4)^{\perp}\]
is the kernel of the operator $\operatorname{Ad}(\mathrm{id}-\rho(a_2)^{-1}\rho(a_1)^{-1})$.

Since we can write the space $D_{\mathrm{id},\mathrm{id},B_1,B_2}K_{1,0}(\frak{g}^4)^{\perp}$ as the intersection :
\[\operatorname{Ad}\big(\frac{\partial}{\partial g}(K_{1,0})(\mathrm{id},\mathrm{id})\big)(\frak{g}^4)^{\perp}\cap \operatorname{Ad}\big(\frac{\partial}{\partial h}(K_{1,0})(\mathrm{id},\mathrm{id})\big)(\frak{g}^4)^{\perp},\] as in \cite{zbMATH03916723}, we deduce that the rank of $D_{\mathrm{id},\mathrm{id}, B_1,B_2}K_{1,0}$ is the codimension of the centralizer of the set $\{\rho(a_1), \rho(a_2)\}$.

In the same way, we compute : 

\[D_{g,h,B_1,B_2}K_{2,0} = \operatorname{Ad}\big(\frac{\partial}{\partial g}(K_{2,0})\big)dg + \operatorname{Ad}\big(\frac{\partial}{\partial B_1}(K_{2,0})\big)dB_1 + \operatorname{Ad}\big(\frac{\partial}{\partial B_2}(K_{2,0})\big)dB_2\]
with \[\frac{\partial}{\partial B_1}(K_{2,0})=\rho(a_1)(\mathrm{id}-B_1 \rho(a_1) B_1^{-1})\]
 and \[\frac{\partial}{\partial B_2}(K_{2,0})=[\rho(a_1),B_1]g\rho(a_2) g ^{-1}(\mathrm{id}-B_2 g\rho(a_2)g^{-1} B_2^{-1}).\]

As in \cite{zbMATH03916723}, the same computations for $K_{2,0}$ allow to deduce that the rank of the map $K$ at the point $(\mathrm{id}, \mathrm{id}, \rho(b_1),\rho(b_2))$ is the codimentsion of the centralizer $\mathrm{Z}_{G}(\rho)$. 
Since $\mathcal{M}(\Gamma,G)$ is the set of classes of representations with discrete centralizer in $G$, we obtain that $K$ is a submersion at $(\mathrm{id}, \mathrm{id}, \rho(b_1),\rho(b_2))$. It proves completely the Lemmas~\ref{Lemma 3.4},~\ref{Lemma 4.3} and~\ref{Lemma 5.3}.

\printbibliography

\end{document}